\documentclass{amsart}
\usepackage{amsmath,amssymb}
\usepackage[margin=35mm]{geometry}

\newtheorem{theorem}{Theorem}
\newtheorem{lemma}{Lemma}
\newtheorem{proposition}{Proposition}
\newtheorem{corollary}{Corollary}

\newtheorem{intro}{Theorem}
\newenvironment{introtheorem}[1]
  {\renewcommand\theintro{#1}\intro}
  {\endintro}
\newcommand{\tO}{\mathtt 0}                     % typewriter 0 for the digit 0
\newcommand{\tL}{\mathtt 1}                     % typewriter 1 for the digit 1
\newcommand{\tZ}{\mathtt 2}                     % typewriter 2 for the digit 2

\DeclareMathOperator{\e}{\mathrm{e}}                % exponential
\allowdisplaybreaks     %allow multiline formulas to be split across pages
\begin{document}
\title{A digit reversal property for Stern polynomials}
\author{Lukas Spiegelhofer}
\address{Institut f\"ur Diskrete Mathematik und Geometrie,
TU Wien\\
Wiedner Hauptstr. 8--10\\
1040 Wien, Austria}
\thanks{The author acknowledges support by the Austrian Science Fund (FWF), projects F5502-N26 and F5505-N26, which are part of the Special Research Program ``Quasi Monte Carlo Methods: Theory and Applications''.}

\begin{abstract}
We consider the following polynomial generalization of Stern's diatomic series:
let $s_1(x,y)=1$, and for $n\geq 1$ set
$s_{2n}(x,y)=s_n(x,y)$ and $s_{2n+1}(x,y)=x\,s_n(x,y)+y\,s_{n+1}(x,y)$.
The coefficient $[x^iy^j]s_n(x,y)$ is the number of hyperbinary expansions of $n-1$ with exactly $i$ occurrences of the digit $\mathtt 2$ and $j$ occurrences of $\mathtt 0$.
We prove that the polynomials $s_n$ are invariant under \emph{digit reversal},
that is, $s_n=s_{n^R}$, where $n^R$ is obtained from $n$ by reversing the binary expansion of $n$.

\end{abstract}
\maketitle

%{{{ Introduction
\section{Introduction}
The Stern sequence (also called Stern's diatomic sequence or Stern--Brocot sequence) $s$ is defined by the recurrence $s_1=1$, $s_{2n}=s_n$ and $s_{2n+1}=s_n+s_{n+1}$.
It was pointed out by Dijkstra~\cite{D1980}, \cite[pp.230--232]{D1982} that this sequence satisfies a symmetry property with respect to reversal of the binary expansion of $n$.
More precisely, for a positive integer $n$ having the proper base-$2$ expansion $n=\bigl(\varepsilon_\nu\varepsilon_{\nu-1}\ldots\varepsilon_0\bigr)_2$ we define
\[
n^R
=\sum_{0\leq i\leq\nu}\varepsilon_{\nu-i}2^i
=\bigl(\varepsilon_0\varepsilon_1\ldots\varepsilon_\nu\bigr)_2.
\]
\begin{introtheorem}{A}[Dijkstra]\label{thm:A}
$s_n=s_{n^R}$.
\end{introtheorem}
Stern's diatomic sequence is closely related to continued fractions (see Stern~\cite{S1858}, Lehmer~\cite{L1929}, Lind~\cite{L1969}, Graham, Knuth, Patashnik~\cite[Exercise~6.50]{GKP1989}):
\newline
if $n=(1^{k_0}0^{k_1}\cdots 1^{k_{r-2}}0^{k_{r-1}}1^{k_r})_2$,
then $s_n$ is the numerator of the continued fraction
$[k_0;k_1,\ldots,k_r]$.
Theorem~\ref{thm:A} is therefore the same as the statement that
$[k_0;k_1,\ldots,k_r]$ and $[k_r;k_{r-1},\ldots,k_0]$ have the same numerator,
which can be proved via continuants.

In~\cite{MS2012}, Morgenbesser and the author proved a digit reversal property for the correlation
\[\gamma_t(\alpha)=\lim_{N\rightarrow\infty}\frac 1N
\sum_{0\leq n<N}\e\bigl(\alpha \sigma_q(n+t)-\alpha \sigma_q(n)\bigr),
\]
where $\e(x)=\exp(2\pi i x)$ and $\sigma_q(n)$ is the sum of digits of $n$ in base $q$ ($q\geq 2$ an integer).
That is, we proved that $\gamma_t(\alpha)=\gamma_{t^R}(\alpha)$,
where the digit reversal is in base $q$.
We note that for the case $q=2$ this statement is a special case of Theorem~\ref{thm:main} below.

In this paper, we wish to give a refinement of Theorem~\ref{thm:A}, concerning \emph{hyperbinary expansions}.
A hyperbinary expansion~\cite{DE2015} of a nonnegative integer $n$ is a sequence $(\varepsilon_{\nu-1},\ldots,\varepsilon_0)\in \{\tO,\tL,\tZ\}^{\nu}$ such that $\sum_{0\leq i<\nu}\varepsilon_i2^i=n$.
We call such an expansion {\it proper} if either $\nu=0$ or $\nu>0$ and $\varepsilon_{\nu-1}\neq \tO$. Since we only work with proper expansions, we will usually omit the prefix ``proper''.
By induction, using the defining recurrence relation,
it is not difficult to see that $s_n$ is the number of hyperbinary expansions of $n-1$ (see also Proposition~\ref{prp:polynomial_rep}, which generalizes this property).
This property is stated as Theorem~5.2 in~\cite{R1990}, which seems to be the earliest (correct) appearance in the literature ---
we note that the statement on the bottom of page 275 of~\cite{C1964} is erroneous, which can be seen considering the case $n=5$. The same statement can be found on page 57 of~\cite{L1969}.
%}}} Introduction
%{{{ Main results
\section{Main results}
Our main theorem generalizes Theorem~\ref{thm:A}.
We define bivariate polynomials $s_n(x,y)$ by
\begin{align*}
s_1(x,y)&=1,\\
s_{2n}(x,y)&=s_n(x,y),\\
s_{2n+1}(x,y)&=x\,s_n(x,y)+y\,s_{n+1}(x,y).
\end{align*}
Note that this definition of Stern polynomials differs from the one given in~\cite{DS2007} and also from the definition in~\cite{KMP2007}.
However, the univariate polynomials $s_n(x,1)$ appear, up to a shift of the index $n$, in the article~\cite{BM2011} by Bates and Mansour: the authors write ``[\ldots]the $n$th term $f(n;q)$ of the $q$-analogue of the Calkin--Wilf sequence is the generating function for the number of hyperbinary expansions of $n$ according to the number of powers that are used exactly twice.'' This should be compared with Proposition~\ref{prp:polynomial_rep} below.
Moreover, $s_n(x,1)$ appears as the special case $t=1$ in~\cite{DE2015} (see~(1.2), (1.3) in that paper).
The bivariate polynomial $s_n(x,y)$, however, appears to be a new object of study.
We list the first few of these polynomials: we have
\begin{align*}
  s_1(x,y) &= 1,&                           s_2(x,y) &= 1,\\
  s_3(x,y) &= x+y,&                         s_4(x,y) &= 1,\\
  s_5(x,y) &= x+xy+y^2,&                    s_6(x,y) &= x+y,\\
  s_7(x,y) &= x^2+xy+y,&                    s_8(x,y) &= 1,\\
  s_9(x,y) &= x+xy+xy^2+y^3,&               s_{10}(x,y) &= x+xy+y^2,\\
  s_{11}(x,y) &= x^2+xy+y^2+x^2y+xy^2,&     s_{12}(x,y) &= x+y,\\
  s_{13}(x,y) &= x^2+xy+y^2+x^2y+xy^2,&     s_{14}(x,y) &= x^2+xy+y,\\
  s_{15}(x,y) &= y+xy+x^3+x^2y,&            s_{16}(x,y) &= 1,\\
  s_{17}(x,y) &= x+xy+xy^2+xy^3+y^4,&       s_{18}(x,y) &= x+xy+xy^2+y^3,\\
  s_{19}(x,y) &= x^2+xy+x^2y+xy^2+y^3+x^2y^2+xy^3,&  s_{20}(x,y) &= x+xy+y^2,\\
  s_{21}(x,y) &= x^2+2x^2y+2xy^2+y^3+x^2y^2+xy^3,&   s_{22}(x,y) &= x^2+xy+y^2\\
              &                                  &               &+x^2y+xy^2,\\
  s_{23}(x,y) &= xy+y^2+x^3+x^2y+xy^2+x^3y+x^2y^2,&  s_{24}(x,y) &=x+y ,\\
  s_{25}(x,y) &= x^2+xy+x^2y+xy^2+y^3+x^2y^2+xy^3,&  s_{26}(x,y) &=x^2+xy+y^2\\
              &                                   &              &+x^2y+xy^2
\end{align*}
and we see the notable identities $s_{11}(x,y)=s_{13}(x,y)$
and $s_{19}(x,y)=s_{25}(x,y)$, where $13=11^R$ and $25=19^R$.
In fact, we have the following symmetry property generalizing Theorem~A.
%{{{
\begin{theorem}\label{thm:main}
Let $n$ be a positive integer. Then
\[ s_n(x,y)=s_{n^R}(x,y). \]
\end{theorem}
%}}}
This theorem can be translated to a statement on hyperbinary expansions.
For integers $i,j\geq 0$ and $t\geq 1$ let $h_{i,j}(t)$ be the number of proper hyperbinary expansions $(\varepsilon_{\nu-1},\ldots,\varepsilon_0)$ of $t-1$ such that $|\{\ell: 0\leq \ell<\nu,\varepsilon_\ell=\tZ\}|=i$ and $|\{\ell:0\leq \ell<\nu,\varepsilon_\ell=\tO\}|=j$.
The following proposition connects the Stern polynomials $s_n(x,y)$ to hyperbinary expansions.
\begin{proposition}\label{prp:polynomial_rep}
Assume that $n\geq 1$ is an integer. Then
\[
\sum_{i,j\geq 0}h_n(i,j)x^iy^j=s_n(x,y).
\]
\end{proposition}
In other words, we have $h_n(i,j)=\bigl[x^iy^j\bigr]s_n(x,y)$,
that is, the polynomial $s_n(x,y)$ encodes the number of hyperbinary expansions of $n-1$ having a given number of $\tZ$s and $\tO$s.
Let us illustrate this proposition by an example.
The polynomial $s_{21}(x,y)$ can be obtained by listing the hyperbinary expansions of $20$:
\begin{align*}
\tL\tZ\tL\tZ&\qquad x^2&\tL\tZ\tZ\tO&\qquad x^2y\\
\tZ\tO\tL\tZ&\qquad x^2y&\tZ\tL\tO\tO&\qquad xy^2\\
\tL\tO\tO\tL\tZ&\qquad xy^2&\tL\tO\tL\tO\tO&\qquad y^3\\
\tZ\tO\tZ\tO&\qquad x^2y^2&\tL\tO\tO\tZ\tO&\qquad xy^3.
\end{align*}
Combining Theorem~\ref{thm:main} and Proposition~\ref{prp:polynomial_rep},
we immediately get the following corollary.
\begin{corollary}
Let $n\geq 1$ and $i,j\geq 0$. We have
\[ h_n(i,j)=h_{n^R}(i,j). \]
\end{corollary}
For example, we list the hyperbinary expansions of
$18=(\tL\tO\tO\tL\tL)_2-1$ (first row) and 
$24=(\tL\tL\tO\tO\tL)_2-1$ (second row):
\[
\begin{array}{ccccccc}
\tL\tO\tO\tL\tO&\tZ\tO\tL\tO&\tL\tZ\tL\tO&\tL\tO\tO\tO\tZ&\tZ\tO\tO\tZ&\tL\tZ\tO\tZ&\tL\tL\tZ\tZ\\
\tL\tL\tO\tO\tO&\tL\tO\tL\tZ\tO&\tL\tO\tL\tL\tZ&\tL\tO\tZ\tO\tO&\tZ\tZ\tO\tO&\tZ\tL\tZ\tO&\tZ\tL\tL\tZ
\end{array}
\]
By the corollary there is a one-to-one correspondence between these expansions.
%}}} Main results
%{{{ Proofs
\section{Proofs}
%{{{ proof of Theorem thm:main
\subsection{Proof of Theorem~\ref{thm:main}}
The main argument, which represents the induction step in the proof of the theorem, is the following lemma.
%{{{ Lemma: main argument about a matrix identity
\begin{lemma}\label{lem:reflection_matrices}
Let
\[
A=\left(\begin{matrix}1&0\\x&y\end{matrix}\right)
\quad\textrm{ and }\quad
B=\left(\begin{matrix}x&y\\0&1\end{matrix}\right)
.\]
If ${}^t\!A$ denotes the transpose of the matrix $A$, the following identities for $1\times 2$-matrices hold.
\begin{equation*}
\begin{array}{c@{}c@{\,}c@{\,}c@{\,}c@{\,}c@{\,}c@{\,}c@{\,}c}
  \left(\begin{matrix}x&y\end{matrix}\right)&AA
&=&
  (-y)&\left(\begin{matrix}x&y\end{matrix}\right)
  &+&
  (y+1)&\left(\begin{matrix}x&y\end{matrix}\right)&A,
\\[1mm]
  \left(\begin{matrix}1&1\end{matrix}\right)&{}^t\!A{}^t\!A
&=&
  (-y)&\left(\begin{matrix}1&1\end{matrix}\right)
  &+&
  (y+1)&\left(\begin{matrix}1&1\end{matrix}\right)&{}^t\!A,
\\[4mm]
  \left(\begin{matrix}x&y\end{matrix}\right)&AB
&=&
  x&\left(\begin{matrix}x&y\end{matrix}\right)
  &+&y&\left(\begin{matrix}x&y\end{matrix}\right)&B,
\\[1mm]
  \left(\begin{matrix}1&1\end{matrix}\right)&{}^t\!A{}^t\!B
&=&
  x&\left(\begin{matrix}1&1\end{matrix}\right)
  &+&y&\left(\begin{matrix}1&1\end{matrix}\right)&{}^t\!B,
\\[4mm]
  \left(\begin{matrix}x&y\end{matrix}\right)&BA
&=&
  y&\left(\begin{matrix}x&y\end{matrix}\right)
  &+&x&\left(\begin{matrix}x&y\end{matrix}\right)&A,
\\[1mm]
  \left(\begin{matrix}1&1\end{matrix}\right)&{}^t\!B{}^t\!A
&=&
  y&\left(\begin{matrix}1&1\end{matrix}\right)
  &+&x&\left(\begin{matrix}1&1\end{matrix}\right)&{}^t\!A,
\\[4mm]
  \left(\begin{matrix}x&y\end{matrix}\right)&BB
&=&
  (-x)&\left(\begin{matrix}x&y\end{matrix}\right)
  &+&(x+1)&\left(\begin{matrix}x&y\end{matrix}\right)&B,
\\[1mm]
  \left(\begin{matrix}1&1\end{matrix}\right)&{}^t\!B{}^t\!B
&=&
  (-x)&\left(\begin{matrix}1&1\end{matrix}\right)
  &+&(x+1)&\left(\begin{matrix}1&1\end{matrix}\right)&{}^t\!B.
\end{array}
\end{equation*}
\end{lemma}
%}}}

\noindent
The proof is by simple calculation and is left to the reader.\qed

\smallskip\noindent
To prove Theorem~\ref{thm:main}, we let $n\geq 1$ be an odd integer.
This is no loss of generality since we can deal with the even case by repeatedly using the relation $s_{2n}(x,y)=s_n(x,y)$.
Let $n=\sum_{i\leq \nu}\varepsilon_i2^i$ be the binary representation of $n$ and $\varepsilon_\nu\neq 0$.
We prove the theorem by induction on $\nu$.
The case $\nu\leq 1$ is trivial, since in this case we have $n^R=n$.
We write
$A(\tO)=\left(\begin{smallmatrix}1&0\\x&y\end{smallmatrix}\right)$ and
$A(\tL)=\left(\begin{smallmatrix}x&y\\0&1\end{smallmatrix}\right)$.
By a simple application of the recurrence relation we have
$\left(\begin{smallmatrix}s_{2n}\\s_{2n+1}\end{smallmatrix}\right)
=A(\tO)
\left(\begin{smallmatrix}s_n\\s_{n+1}\end{smallmatrix}\right)$
and
$\left(\begin{smallmatrix}s_{2n+1}\\s_{2n+2}\end{smallmatrix}\right)
=A(\tL)
\left(\begin{smallmatrix}s_n\\s_{n+1}\end{smallmatrix}\right)$
for all $n\geq 1$.
Since $n$ is odd and $s_1(x,y)=s_2(x,y)=1$, it follows from these identities that
\begin{equation}\label{eqn:transition_matrices}
s_n(x,y)
=
\left(\begin{matrix}x&y\end{matrix}\right)
A(\varepsilon_1)\cdots A(\varepsilon_{\nu-1})
\left(\begin{matrix}1\\1\end{matrix}\right).
\end{equation}
The statement of the theorem is equivalent to the assertion that
\begin{equation}\label{eqn:reflection_matrix_product}
\left(\begin{matrix}x&y\end{matrix}\right)
A(\varepsilon_1)\cdots A(\varepsilon_{\nu-1})
\left(\begin{matrix}1\\1\end{matrix}\right)
=
\left(\begin{matrix}1&1\end{matrix}\right)
{}^t\!A(\varepsilon_1)\cdots {}^t\!A(\varepsilon_{\nu-1})
\left(\begin{matrix}x\\y\end{matrix}\right)
\end{equation}
for all $\nu\geq 1$ and all finite sequences $(\varepsilon_1,\ldots,\varepsilon_{\nu-1})$ in $\{\tO,\tL\}$.
We prove the identity~\eqref{eqn:reflection_matrix_product} by induction on $\nu$, using Lemma \ref{lem:reflection_matrices}.
This identity is obvious for $\nu\leq 2$.
For $\nu>2$ we have four cases, corresponding to the four possible values of $(\varepsilon_1,\varepsilon_2)$.
By Lemma \ref{lem:reflection_matrices}, in each of the four cases there exist coefficients $\alpha$ and $\beta$ such that
\[
  \left(\begin{matrix}x&y\end{matrix}\right)A(\varepsilon_1)A(\varepsilon_2)
=
  \alpha\left(\begin{matrix}x&y\end{matrix}\right)
+
  \beta\left(\begin{matrix}x&y\end{matrix}\right)A(\varepsilon_2)
\]
and
\[  
  \left(\begin{matrix}1&1\end{matrix}\right){}^t\!A(\varepsilon_1){}^t\!A(\varepsilon_2)
=
  \alpha\left(\begin{matrix}1&1\end{matrix}\right)
+
  \beta\left(\begin{matrix}1&1\end{matrix}\right){}^t\!A(\varepsilon_2)
.
\]
Applying the induction hypothesis (\ref{eqn:reflection_matrix_product})
to the sequences $(\varepsilon_2,\ldots,\varepsilon_{\nu-1})$ and
$(\varepsilon_3,\ldots,\varepsilon_{\nu-1})$ we obtain the statement of the theorem.\qed

%}}} Proof of thm:main
%{{{ Proof of prp:polynomial_rep
\subsection{Proof of Proposition~\ref{prp:polynomial_rep}}
In order to prove Proposition~\ref{prp:polynomial_rep}, we use the following recurrence for $h_n(i,j)$ (see~\cite{DKS2016}).
%{{{ Proposition
\begin{proposition}\label{prp:hyper_rec}
Let $n\geq 1$. Then
\begin{equation}\label{eqn:hyperbinary_recurrence}
\begin{array}{lcll}
h_1(0,0)&=&1,&\\
h_1(i,j)&=&0&\mbox{for }(i,j)\neq (0,0),\\
h_{2n}(i,j)&=&h_n(i,j)&\mbox{for }i,j\geq 0,\\
h_{2n+1}(i,0)&=&h_n(i-1,0)&\mbox{for }i\geq 1,\\
h_{2n+1}(0,j)&=&h_{n+1}(0,j-1)&\mbox{for }j\geq 1,\\
h_{2n+1}(i,j)&=&h_n(i-1,j)+h_{n+1}(i,j-1)&\mbox{for }i,j\geq 1.
\end{array}
\end{equation}
Moreover, $h_t(0,0)=0$ if $t$ is not a power of $2$.
\end{proposition}
%}}}
%{{{
\begin{proof}
The first two lines follow from the fact that the empty tuple $()$ is the only hyperbinary expansion of $0$. 

The hyperbinary expansions of $2n-1$ are in bijection with the hyperbinary expansions of $n-1$ by deleting the lowest digit, which is a $\tL$. This explains the third line of~\eqref{eqn:hyperbinary_recurrence}.

A hyperbinary expansion of $2n$ without $\tO$ necessarily ends with the digit $\tZ$, and the bijection is by deleting this digit. This gives line number four.

There is exactly one hyperbinary expansion of $2n$ without $\tZ$, and it ends with $\tO$. The same argument applies here.

The sixth line is a combination of arguments as above.

Finally, integers having the binary expansion $(\tL\cdots\tL)$ are the only ones without $\tO$ and $\tZ$, which proves the last line.
\end{proof}
%}}}
In order to prove Proposition~\ref{prp:polynomial_rep}, we proceed by induction.
The identity is trivial for $n=1$.
Assume that $n=2u$ for some $u\geq 1$. We have
$h_n(i,j)=h_u(i,j)$ and $s_n(x,y)=s_u(x,y)$.
If $n=2u+1$, we have
\begin{align*}
\sum_{\substack{i\geq 0\\j\geq 0}}h_n(i,j)x^iy^j&=h_n(0,0)
+\sum_{j\geq 1}h_n(0,j)y^j
+\sum_{i\geq 1}h_n(i,0)x^i
+\sum_{i,j\geq 1}h_n(i,j)x^iy^j\\
&=\sum_{j\geq 1}h_{u+1}(0,j-1)y^j
+\sum_{i\geq 1}h_u(i-1,0)x^i
\\&+\sum_{i,j\geq 1}\bigl(h_u(i-1,j)+h_{u+1}(i,j-1)\bigr)x^iy^j
\\&+\sum_{\substack{i\geq 1\\j\geq 0}}h_u(i-1,j)x^iy^j
+\sum_{\substack{i\geq 0\\j\geq 1}}h_u(i,j-1)x^iy^j\\
&=x s_u(x,y)+y s_{u+1}(x,y)=s_n(x,y)
,
\end{align*}
which proves the desired assertion.
\qed
%}}} Proof of Proposition prp:polynomial_rep
%}}} Proofs
%{{{ Bibliography

%}}} Bibliography

\end{document}